\documentclass{amsart}

\usepackage{amsmath,amsthm,amssymb}

\theoremstyle{plain}

\newtheorem{definition}{Definition}
\newtheorem{thm}[definition]{Theorem}

\newtheorem{lem}[definition]{Lemma}
\newtheorem{cor}[definition]{Corollary}
\newtheorem{rei}[definition]{Example}

\newcommand{\diag}{{\rm diag}}
\newcommand{\Mat}{{\rm Mat}}
\newcommand{\Det}{{\rm Det}}
\newcommand{\rdet}{{\rm rdet}}
\newcommand{\Tr}{{\rm Tr}}
\newcommand{\sgn}{{\rm sgn}}

\begin{document}
\title[]{Capelli elements of the group algebra}
\author[N. Yamaguchi]{Naoya Yamaguchi}
\date{\today}

\maketitle

\begin{abstract}
Inspired by the Capelli identities for group determinants obtained by T\^oru Umeda, 
we give a basis of the center of the group algebra of any finite group by using Capelli identities for irreducible representations.

The Capelli identities for irreducible representations are modifications of the Capelli identity. 
These identities lead to Capelli elements of the group algebra. 
These elements construct a basis of the center of the group algebra. 
\end{abstract}

\section{\bf{INTRODUCTION}}
The Capelli identity is analogous to the product formula for the determinant in the Weyl algebra. 
The identity leads to the Capelli element. 
It is known that the Capelli element is a central element in the universal enveloping algebra of $\mathfrak{gl}_{n}$.

On the other hand, An Huang gave Capelli-type identities associated with the quaternions and the octonions \cite{H}. 
Inspired by his results, T\^oru Umeda gave Capelli identities for group determinants \cite{U}. 
There are Capelli identities for irreducible representations in the background of the Capelli identities for group determinants.

In this paper, 
we give a basis of the center of the group algebra of any finite group by using Capelli identities for irreducible representations. 
These identities lead to Capelli elements of the group algebra. 
These elements construct a basis.

First, we explain our motivation.

\subsection{Motivation}

Let $G$ be a finite group, 
$\widehat{G}$ a complete set of irreducible representations of $G$ over $\mathbb{C}$, 
$\mathbb{C}G = \left\{ \sum_{g \in G} x_{g} g \: \vert \: x_{g} \in \mathbb{C} \right\}$ the group algebra, 
and $Z(\mathbb{C}G)$ the center of $\mathbb{C}G$. 
The following theorem is easily proven from Schur's orthogonal relations.

\begin{thm}\label{thm:1.1.1}
Let $\chi_{\varphi}$ be the character of $\varphi \in \widehat{G}$. 
The set 
$$
\left\{ \Tr{ \left( \sum_{g \in G} \varphi(g) g \right) } \: \vert \: \varphi \in \widehat{G} \right \} 
= \left\{ \sum_{g \in G} \chi_{\varphi}(g) g \: \vert \: \varphi \in \widehat{G} \right\}
$$
is a basis of $Z(\mathbb{C}G)$ 
where we omit the numbering of the element of the basis. 
\end{thm}

At this point, we have a simple question. 
Is the set 
$\left\{ \det{ \left( \sum_{g \in G} \varphi(g) g \right) } \: \vert \: \varphi \in \widehat{G} \right\}$ 
a basis of $Z(\mathbb{C}G)$?
Our main result gives an answer.

\subsection{\bf{Main result}}

Let $z$ be a complex variable, 
$|G|$ the order of $G$, 
$m = \deg{\varphi}$, 
$\alpha = \frac{|G|}{m}$, 
$u_{i}(z) = \alpha(m - i) - z$, 
$u^{(i)}(z) = u_{m}(z) u_{m - 1}(z) \cdots u_{m - i + 1}(z)$, 
$\det{}$ the column determinant, 
and the Capelli element for $\varphi$ of the group algebra 
$$
\overline{C}^{\varphi}(z) = \det{\left( \sum_{g \in G} \varphi(g) g + \alpha \left( 
\begin{bmatrix} 
m - 1 & & & \\ 
 & m - 2 & & \\
& & \ddots & \\ 
& & & 0 
\end{bmatrix} 
\right) - z I_{m} \right)} \in \mathbb{C}[z] \otimes \mathbb{C}G. 
$$

Then we can prove the following relation.

\begin{thm}\label{thm:1.2.1}
We have 
$$
\overline{C}^{\varphi}(z) = u^{(m)}(z) + \Tr{\left( \sum_{g \in G} \varphi(g) g \right)} u^{(m - 1)}(z)
$$
\end{thm}

The above relation leads to the following corollary.

\begin{cor}\label{cor:1.2.2}
Suppose $k_{\varphi} \in \mathbb{C}$ such that $u^{(m - 1)}(k_{\varphi}) \neq 0$. 
Then, 
$$
\left\{ \overline{C}^{\varphi}(k_{\varphi}) \: \vert \: \varphi \in \widehat{G} \right\}
$$
is a basis of $Z(\mathbb{C}G)$. 
\end{cor}

This is our answer. 
We provide some sections for the details.

\section{\bf{Capelli identity and Capelli element}}

Here, we review the Capelli identity and the Capelli element.

\subsection{Column determinant}
First, we explain the column determinant. 
Let $R$ be an associative algebra.

\begin{definition}[Column determinant]\label{def:2.1.1}
Let $A = (a_{i j})_{1 \leq i, j \leq m} \in \Mat(m, R)$. 
We define the column determinant of $A$ by 
$$
\det{A} = \sum_{\sigma \in S_{m}} \sgn(\sigma) a_{\sigma(1) 1} a_{\sigma(2) 2} \cdots a_{\sigma(m) m}. 
$$
\end{definition}

Hence, we have 
$
\det{
\begin{bmatrix}
a & b \\ 
c & d
\end{bmatrix}
} = ad - cb$.

\subsection{Weyl algebra}
The Capelli identity is analogous to the product formula for the determinant in the Weyl algebra. 
Next, we explain the Weyl algebra $\mathbb{C}[x_{ij}, \partial_{kl} \: \vert \: 1 \leq i, j, k, l \leq m]$.

Let $x_{ij} \: (1 \leq i, j \leq m)$ be variables and 
$\partial_{ij} = \frac{\partial}{\partial x_{ij}} \: (1 \leq i, j \leq m)$ partial differential operators. 
We assume that these variables and operators are related as follows.

For all $1 \leq i, j, k, l \leq m$, we have 
\begin{align*}
[x_{ij}, x_{kl}] = 0, \quad 
[\partial_{ij}, \partial_{kl}] = 0, \quad 
[\partial_{i j}, x_{kl}] = \alpha \delta_{i k} \delta_{j l} 
\end{align*}
where $\delta$ is the Kronecker delta. 
Usually, we take $\alpha = 1$. 
Here, we will not assume that $\alpha = 1$. 
The Weyl algebra is generated by these variables and operators.

\subsection{Capelli identity}

Next, we explain the Capelli identity. 
Let 
\begin{align*}
&X = \left( x_{i j} \right)_{1 \leq i \leq m, 1 \leq j \leq m}, 
&&\partial = \left( \partial_{ij} \right)_{1 \leq i \leq m, 1 \leq j \leq m}, \quad \\ 
&\varPi = {}^{t}\!X \partial, 
&&\natural_{m} = \diag(m-1, m-2, \ldots, 0) 
\end{align*}

The Capelli identity is as follows.

\begin{thm}[Capelli identity]\label{thm:2.3.1}
We have 
\begin{align*}
\det{\left( \varPi + \alpha \natural_{m} \right)} = \det{X} \det{\partial}. 
\end{align*}
\end{thm}

\begin{rei}\label{rei:2.3.2}
Let $m = 2$ and $\alpha = 1$. 
We have 
\begin{align*}
\det{
\begin{bmatrix} 
x_{11} \partial_{11} + x_{21} \partial_{21} + 1 & x_{11} \partial_{12} + x_{21} \partial_{22} \\ 
x_{12} \partial _{11} + x_{22} \partial_{21} & x_{21} \partial_{12} + x_{22} \partial_{12} 
\end{bmatrix}
} 
= \det{
\begin{bmatrix} 
x_{11} & x_{12} \\ 
x_{21} & x_{22} 
\end{bmatrix}
} 
\det{
\begin{bmatrix}
\partial_{11} & \partial_{12} \\ 
\partial_{21} & \partial_{22} 
\end{bmatrix}
}
\end{align*}
\end{rei}

\subsection{Capelli element}

The Capelli element is a charactrestic polynomial of $\varPi$. 
Let $z$ be a variable.

\begin{definition}[Capelli element]\label{def:2.4.1}
We define the Capelli element $C(z)$ by 
$$
C(z) = \det{\left( \varPi + \alpha \natural_{m} - z I_{m} \right)}. 
$$
\end{definition}

The Capelli identity is conjugation invariant.

\begin{thm}\label{thm:2.4.2}
For all $P \in GL(m, \mathbb{C})$, we have 
\begin{align*}
\det{\left( P \varPi P^{-1} + \alpha \natural_{m} - z I_{m} \right)} = C(z).
\end{align*}
\end{thm}

The following theorem plays an important role in what follows.

\begin{thm}\label{thm:2.4.3}
For all $1 \leq i, j \leq m$, we have 
\begin{align*}
[\varPi_{i j}, C(z)] = 0.
\end{align*}
\end{thm}

Theorem~$\ref{thm:2.4.2}$ and $\ref{thm:2.4.3}$ are obtained from only the following relations.

For all $1 \leq i, j, k, l \leq m$, 
$$
[\varPi_{i j}, \varPi_{k l}] = \alpha (\delta_{j k} \varPi_{i l} - \delta_{i l} \varPi_{k j}). 
$$

\section{\bf{Capelli identity for irreducible representations}}
Here, we explain the Capelli identities for irreducible representations.

Let $G$ be a finite group, 
$x_{g} \: (g \in G)$ variable and 
$\partial_{g} = \frac{\partial}{\partial x_{g}} \: (g \in G)$ partial differential operator. 
We assume that the following relations hold.

For all $g, h \in G$, 
\begin{align*}
[x_{g}, x_{h}] = 0, \quad [\partial_{g}, \partial_{h}] = 0, \quad [\partial_{g}, x_{h}] = \delta_{g h}. 
\end{align*}
Then, we have the Weyl algebra $\mathbb{C}[x_{g}, \partial_{h}]$. 
Next, we construct Weyl subalgebras of the Weyl algebra by using irreducible unitary representations of $G$.

Let $|G|$ be the cardinality of the set $G$ (that is, $|G|$ is the order of the group $G$), 
$\varphi$ a unitary matrix form of an irreducible representation of $G$, 
\begin{align*}
\alpha = \frac{|G|}{m}, \quad 
X^{\varphi} = \sum_{g \in G} \overline{\varphi(g)} x_{g}, \quad 
\partial^{\varphi} = \sum_{g \in G} \varphi(g) \partial_{g}, {\rm and} \quad
\varPi^{\varphi} = {}^{t}\!X^{\varphi} \partial^{\varphi} 
\end{align*}
where $\overline{\varphi(g)}$ is the complex conjugate matrix of $\varphi(g)$. 
Then, we have the following relations.

For all $1 \leq i, j, k, l \leq m$, 
\begin{align*}
[X^{\varphi}_{i j}, X^{\varphi}_{k l}] = 0, \quad [\partial^{\varphi}_{i j}, \partial^{\varphi}_{k l}] = 0, \quad [\partial^{\varphi}_{i j}, X^{\varphi}_{k l}] = \alpha \delta_{i k} \delta_{j l}. 
\end{align*}

This leads us to the following identity.

\begin{thm}[Capelli identity for irreducible representations]\label{thm:3.1.1}
We have 
\begin{align*}
\det{\left( \varPi^{\varphi} + \alpha \natural_{m} \right)} = \det{X^{\varphi}} \det{\partial^{\varphi}}. 
\end{align*}
\end{thm}

Let $C^{\varphi}(z) = \det{\left( \varPi^{\varphi} + \alpha \natural_{m} - z I_{m} \right)}$ be the Capelli element. 
From Theorem~$\ref{thm:2.4.2}$, the Capelli element is invariant under a change of a matrix form of the irreducible representation. 
This enables us to redefine the Capelli element as follows.

\begin{definition}[Capelli element for irreducible representations]\label{def:3.1.2}
Let $\varphi \in \widehat{G}$ and $m = \deg{\varphi}$. 
We define $C^{\varphi}(z)$ by 
$$
C^{\varphi}(z) = \det{\left( \varPi^{\varphi} + \alpha \natural_{m} - z I_{m} \right)}
$$
We call $C^{\varphi}(z)$ the Capelli element for $\varphi$. 
\end{definition}

\section{\bf{Capelli element of the group algebra}}

Let $\mathbb{C}G = \left\{ \sum_{g \in G} x_{g} g \: \vert \: x_{g} \in \mathbb{C} \right\}$ the group algebra of $G$, 
$\widetilde{G}$ a complete set of irreducible unitary matrix representations of $G$, 
$\varphi \in \widetilde{G}$, 
and 
$$
E^{\varphi} = \sum_{g \in G} \varphi(g) g \in \Mat(\deg{\varphi}, \mathbb{C}G). 
$$
From Schur's orthogonal relations, we have the following lemmas.

\begin{lem}\label{lem:4.1.1}
$\{ E^{\varphi}_{i j} \: \vert \: 1 \leq i, j \leq \deg{\varphi}, \varphi \in \widetilde{G} \}$ is a basis of $\mathbb{C}G$. 
\end{lem}

\begin{lem}\label{lem:4.1.2}
Let $\varphi, \psi \in \widetilde{G}$, where $\varphi$ is not equivalent to $\psi$. 
For all $1 \leq i, j \leq \deg{\varphi}$ and $1 \leq s, t \leq \deg{\psi}$, we have 
\begin{align*}
E^{\varphi}_{ij} E^{\varphi}_{kl} = \alpha_{\deg{\varphi}} \delta_{jk} E_{il}, \quad E^{\varphi}_{ij} E^{\psi}_{st} = 0. 
\end{align*}
In particular, we have
\begin{align}
[E^{\varphi}_{i j}, E^{\varphi}_{k l}] &= \alpha_{\deg{\varphi}} (\delta_{j k} E_{i l} - \delta_{i l} E_{k j}), \\ 
[E^{\varphi}_{i j}, E^{\psi}_{k l}] &= 0. 
\end{align}
\end{lem}

Let 
$$
\overline{C}^{\varphi}(z) = \det{\left( E^{\varphi} + \alpha_{m} \natural_{m} - z I_{m} \right)} \in \mathbb{C}[z] \otimes \mathbb{C}G. 
$$ 
Recall that Theorem~$\ref{thm:2.4.2}$ and Theorem~$\ref{thm:2.4.3}$ are obtained from only the relations $[\varPi_{i j}, \varPi_{k l}] = \alpha (\delta_{j k} \varPi_{i l} - \delta_{i l} \varPi_{k j})$. 
Hence, $\overline{C}^{\varphi}(z)$ is conjugation invariant from the relations~$(1)$, 
and we have 
\begin{align}
[E^{\varphi}_{ij}, \overline{C}^{\varphi}(z)] = 0
\end{align}
for any $1 \leq i, j \leq \deg{\varphi}$.

Using the above conjugation invariance, we redefine $\overline{C}^{\varphi}(z)$.

\begin{definition}[Capelli element of the group algebra]\label{def:4.1.3}
Let $\varphi \in \widehat{G}$. 
We define the Capelli element for $\varphi$ of the group algebra by 
$$
\overline{C}^{\varphi}(z) = \det{\left( E^{\varphi} + \alpha_{m} \natural_{m} - z I_{m} \right)}. 
$$ 
\end{definition}

From Lemma~$\ref{lem:4.1.1}$ and relations~$(1)$ and $(2)$, we can prove the following Lemma.

\begin{lem}\label{lem:4.1.4}
For all $\varphi \in \widehat{G},\ \overline{C}^{\varphi}(z) \in Z(\mathbb{C}G)$. 
That is, $\overline{C}^{\varphi}(z)$ is a central element of the group algebra. 
\end{lem}

Let $u_{i}(z) = \alpha_{m} (m - i) - z$, 
$u^{(i)}(z) = u_{m}(z) u_{m - 1}(z) \cdots u_{m - i + 1}(z)$, 
$E^{\varphi}_{ij}(u_{i}(z)) = E^{\varphi}_{ij} + \delta_{ij} u_{i}(z)$, 
and $[m] = \{ 1, 2, \ldots, m \}$. 
The following is the main theorem.

\begin{thm}\label{thm:4.1.5}
We have 
$$
\overline{C}^{\varphi}(z) = u^{(m)}(z) + \Tr{\left( E^{\varphi} \right)} u^{(m - 1)}(z)
$$
\end{thm} 
\begin{proof}
From the definition of $\overline{C}^{\varphi}(z)$, we have 
\begin{align*}
\overline{C}^{\varphi}(z) 
&= \sum_{\sigma \in S_{m}} \sgn(\sigma) E^{\varphi}_{\sigma(1) 1}(u_{1}(z)) E^{\varphi}_{\sigma(2) 2}(u_{2}(z)) \cdots E^{\varphi}_{\sigma(m) m}(u_{m}(z)) \\ 
&= \sum_{\sigma \in S_{m}} \sgn(\sigma) (E^{\varphi}_{\sigma(1) 1} + \delta_{\sigma(1) 1} u_{1}(z)) \cdots (E^{\varphi}_{\sigma(m) m} + \delta_{\sigma(m) m} u_{m}(z)) \\ 
&= \sum_{\sigma \in S_{m}} \sgn(\sigma) \sum_{T \subset [m]} \prod_{t \in T} E^{\varphi}_{\sigma(t) t} \prod_{s \in [m] \setminus T} \delta_{\sigma(s) s} u_{s}(z). 
\end{align*}
Let $T = \{ t_{1} < t_{2} < \ldots < t_{|T|} \}$ and $[m] \setminus T = \{ s_{1}, s_{2}, \ldots, s_{m - |T|} \}$. 
From Lemma~$\ref{lem:4.1.2}$, we have 
\begin{align*}
&\sum_{\sigma \in S_{m}} \sgn(\sigma) \sum_{T \subset [m]} \prod_{t \in T} E^{\varphi}_{\sigma(t) t} \prod_{s \in [m] \setminus T} \delta_{\sigma(s) s} u_{s}(z) \\ 
&\quad = \alpha_{m}^{|T| - 1} \sum_{\sigma \in S_{m}} \sgn(\sigma) \delta_{\sigma(t_{2}) t_{1}} \delta_{\sigma(t_{3}) t_{2}} \cdots \delta_{\sigma(t_{|T|}) t_{|T| - 1}} E^{\varphi}_{\sigma(t_{1}) t_{|T|}} \prod_{s \in [m] \setminus T} \delta_{\sigma(s) s} u_{s}(z) \\ 
&\quad = (- \alpha_{m})^{|T| - 1} E^{\varphi}_{t_{|T|} t_{|T|}} \prod_{s \in [m] \setminus T} u_{s}(z). 
\end{align*}
Therefore, there exists $a_{i} \in \mathbb{C}[z] \: (1 \leq i \leq m)$ such that 
\begin{align*}
\overline{C}^{\varphi}(z) = \sum_{i = 1}^{m} a_{i} E^{\varphi}_{ii} + u^{(m)}(z). 
\end{align*}
We show that $a_{p} = a_{q}$ for all $p, q \in [m]$. 
From Lemma~$\ref{lem:4.1.2}$, we have 
\begin{align*}
E^{\varphi}_{pq} \overline{C}^{\varphi}(z) 
&= E^{\varphi}_{pq} \left( \sum_{i = 1}^{m} a_{i} E^{\varphi}_{ii} + u^{(m)}(z) \right) \\ 
&= \sum_{i = 1}^{m} a_{i} \delta_{iq} \alpha_{m} E^{\varphi}_{pi} + u^{(m)}(z) E^{\varphi}_{pq} \\ 
&= a_{q} \alpha_{m} E^{\varphi}_{pq} + u^{(m)}(z) E^{\varphi}_{pq}, \\ 
\overline{C}^{\varphi}(z) E^{\varphi}_{pq} &= \left( \sum_{i = 1}^{m} a_{i} E^{\varphi}_{ii} + u^{(m)}(z) \right) E^{\varphi}_{pq} \\ 
&= \sum_{i = 1}^{m} a_{i} \delta_{ip} \alpha_{m} E^{\varphi}_{iq} + u^{(m)}(z) E^{\varphi}_{pq} \\ 
&= a_{p} \alpha_{m} E^{\varphi}_{pq} + u^{(m)}(z) E^{\varphi}_{pq}. 
\end{align*}
From Lemma~$\ref{lem:4.1.4}$, we have $a_{p} = a_{q}$ for all $p, q \in [m]$. 
We calculate $a_{1}$. 
From $\overline{C}^{\varphi}(z) = \sum_{T \subset [m]} (- \alpha_{m})^{|T|-1} E^{\varphi}_{t_{|T|} t_{|T|}} \prod_{s \in [m] \setminus T} u_{s}(z)$, 
we have 
\begin{align*}
a_{1} E^{\varphi}_{11} &= (- \alpha_{m})^{\{ 1 \} - 1} E^{\varphi}_{11} \prod_{s \in [m] \setminus \{ 1 \}} u_{s}(z) \\ 
&= u^{(m-1)}(z) E^{\varphi}_{11}. 
\end{align*}
This completes the proof. 
\end{proof}

In addition, we have the following corollary.

\begin{cor}\label{cor:4.1.6}
Suppose $k_{\varphi} \in \mathbb{C}$ such that $u^{(m - 1)}(k_{\varphi}) \neq 0$. 
Then, 
$$
\left\{ \overline{C}^{\varphi}(k_{\varphi}) \: \vert \: \varphi \in \widehat{G} \right\}
$$
is a basis of $Z(\mathbb{C}G)$. 
\end{cor}

\section{\bf{Relationship between column, row and double determinant}}
In this last section, we explain the relationship between column, row, and double determinants. 
The row and double determinants are as follows.

\begin{definition}[Row determinant]\label{def:..}
Let $A = (a_{ij})_{1 \leq i, j \leq m} \in \Mat(m, R)$. 
We define the row determinant of $A$ is defined as 
$$
\rdet{A} = \sum_{\sigma \in S_{m}} \sgn(\sigma) a_{1 \sigma(1)} a_{2 \sigma(2)} \cdots a_{m \sigma(m)}. 
$$
\end{definition}

\begin{definition}[Double determinant]\label{def:..}
Let $A = (a_{ij})_{1 \leq i, j \leq m} \in \Mat(m, R)$. 
The double determinant of $A$ is defined as
$$
\Det{A} = \frac{1}{m!} \sum_{\sigma, \tau \in S_{m}} \sgn(\sigma \tau) a_{\sigma(1) \tau(1)} a_{\sigma(2) \tau(2)} \cdots a_{\sigma(m) \tau(m)}. 
$$
\end{definition}

Reference \cite{M} describes that the relationship between column, row, and double determinants. 
Let 
\begin{align*}
\natural^{*} = 
\begin{bmatrix} 
0 & & & \\ 
 & 1 & & \\ 
 & & \ddots & \\ 
& & & m-1 
\end{bmatrix}, \quad 
\natural_{\sigma} = 
\begin{bmatrix}
\sigma(m) & & & \\ 
& \sigma(m-1) & & \\ 
& & \ddots & \\ 
& & & \sigma(1) 
\end{bmatrix} \quad (\sigma \in S_{m}), 
\end{align*}
and $E \in \Mat(m, R)$, where we assume that $\left[E_{ij}, E_{kl} \right] = \delta_{jk} E_{il} - \delta_{il} E_{kj}$ for all $1 \leq i, j, k, l \leq m$. 
We can prove the follwoing theorem.

\begin{thm}[\cite{M}]\label{thm:..}
For all $\sigma \in S_{m}$, we have 
\begin{align*}
\det{\left( E + \natural_{m} - z I_{m} \right)} 
&= \rdet{\left( E + \natural^{*} - z I_{m} \right)} \\ 
&= \Det{ \left( E + \natural_{\sigma} - (z + 1) I_{m} \right)}. 
\end{align*} 
\end{thm}

The above has the following implication.

\begin{cor}\label{cor:..}
Let $\overline{C}^{\varphi}(z)$ be the Capelli element for $\varphi$ of the group algebra. 
For all $\sigma \in S_{m}$, we have 
\begin{align*}
\overline{C}^{\varphi}(z) 
&= \rdet{\left( E^{\varphi} + \alpha_{m} \natural^{*} - z I_{m} \right)} \\ 
&= \Det{ \left( E^{\varphi} + \alpha_{m} \natural_{\sigma} - (z + 1) I_{m} \right)}. 
\end{align*} 
\end{cor}

\clearpage

\thanks{\bf{Acknowledgments}}
I am deeply grateful to Prof. Hiroyuki Ochiai, Prof. T\^oru Umeda and Prof. Minoru Itoh who provided the helpful comments and suggestions. 
Also, I would like to thank my colleagues in the Graduate School of Mathematics of Kyushu University, 
in particular Yuka Suzuki for comments and suggestions. 
This work was supported by a grant from the Japan Society for the Promotion of Science (JSPS KAKENHI Grant Number 15J06842).

\medskip
\begin{flushleft}
Naoya Yamaguchi\\
Graduate School of Mathematics\\
Kyushu University\\
Nishi-ku, Fukuoka 819-0395 \\
Japan\\
n-yamaguchi@math.kyushu-u.ac.jp
\end{flushleft}


\begin{thebibliography}{99}


\bibitem{}
ARTIN, Emil. {\em Geometric algebra}. Courier Dover Publications, 2016. 
\bibitem{}
ASLAKSEN, Helmer. Quaternionic determinants. {\em The Mathematical Intelligencer}, 1996, 18.3: 57--65. 
\bibitem{H}
HUANG, An. Noncommutative multiplicative norm identities for the quaternions and the octonions. {\em arXiv preprint arXiv:1102.2657}, 2011. 
\bibitem{benjamin} 
STEINBERG, Benjamin. {\em Representation theory of finite groups: an introductory approach}. Springer Science \& Business Media, 2011. 
\bibitem{M} 
ITOH, Minoru; UMEDA, T\^oru. On central elements in the universal enveloping algebras of the orthogonal Lie algebras. Compositio Mathematica, 2001, 127.03: 333--359. 
\bibitem{}
UMEDA, T\^oru. On the proof of the Capelli identities. {\em Funkcialaj Ekvacioj}, 2008, 51.1: 1--15.
\bibitem{U}
UMEDA, T\^oru. Remarks on the Capelli identities for reducible modules, preprint(2016). 



\end{thebibliography}
\end{document}